\documentclass[11pt]{article}
\usepackage{amsmath,amssymb,amsfonts,amsthm}
\usepackage{float}
\numberwithin{equation}{section}
\newtheorem{theorem}{Theorem}[section]
\newtheorem{definition}{Definition}[section]
\newtheorem{lemma}[theorem]{Lemma}
\newtheorem{proposition}[theorem]{Proposition}
\newtheorem{corollary}[theorem]{Corollary}

\begin{document}
\begin{center}
{\Large{\textbf{The Derivative Degree Sequences of Finite Simple Connected Graphs are Parking Functions}}} 
\end{center}
\vspace{0.5cm}
\large{\centerline{(Johan Kok)\footnote {\textbf {Affiliation of author(s):}\\
\noindent Johan Kok (Tshwane Metropolitan Police Department), City of Tshwane, Republic of South Africa\\
e-mail: kokkiek2@tshwane.gov.za}}
\vspace{0.5cm}
\begin{abstract}
\noindent Parking functions are well researched and interesting results are found in the listed references and more. Some introductory results stemming from application to degree sequences of simple connected graphs are provided in this paper. Amongst others, the result namely, that a derivative degree sequence,\\ \\
$d_d(G) \in \Bbb D_d(G) = \{(\lceil\frac{d(v_1)}{\ell}\rceil, \lceil\frac{d(v_2)}{\ell}\rceil, \lceil\frac{d(v_3)}{\ell}\rceil, ..., \lceil\frac{d(v_n)}{\ell}\rceil)| \ell = d(v_i), \forall i,$ with $d(v_i) \geq 2\},$ \\ \\of a simple connected graph $G$ is a parking function, is presented. We also introduce the concept of \emph{looping degree sequences} and the \emph{looping number}, $\xi(G)$. Four open problems are proposed as well.
\end{abstract}
\noindent {\footnotesize \textbf{Keywords:} Parking functions, Derivative degree sequence, Looping degree sequences, Looping number $\xi(G)$, \\ Recursive parking function.}\\ \\
\noindent {\footnotesize \textbf{AMS Classification Numbers:} 05C07, 05C38, 05C75, 05C85} 
\section{Introduction} 
Let the simple connected graph $G = (V, E)$ have vertices $V = \{v_1, v_2, v_3, ..., v_n\}.$ Allow each vertex $v_i$ to associate itself with a singular value $d(v_i) =  j\in \{0, 1, 2, 3, ..., n-1\},$ as its preferred value. Allow for $n$ parking spaces $p_1, p_2, p_3, .., p_n$ and allow the the vertices to stream in randomly under the rule that a vertex $v_i$ only occupies a \emph{parking space} if its preferred value $j$ has \emph{a vacant parking space} $p_j$ or has $p_{k,(j+1) \leq k\leq n}$ vacant, else vertex $v_i$ \emph{cliffs}. It is known that if $\alpha = (a_1, a_2, a_3, ..., a_n) \in \Bbb P^n$ and $b_1 \leq b_2 \leq ... \leq b_n$ is the increasing representation of $\alpha$ then $\alpha$ is a parking function if and only if $b_i \leq i.$ It is also known that every permutation of the entries of a parking function is a parking function. So the converse holds namely, if a sequense of integers say, $(a_1, a_2, a_3, ..., a_n)$ is not a parking function then no permutation thereof is a parking function. In addition the strict definition it will be relaxed so that the \emph{default} preferred value \emph{zero} is allowed. So we allow for cases $\alpha^* = (0, a_1, a_2, a_3, ..., a_n)$ to be considered.\\ \\
For ease of reference let us state a corollary done similarly by Stanley, R.P. in $[11]$.
\begin{corollary}
Let $\alpha = (a_1, a_2, a_3, ..., a_n) \in \Bbb P^n$ and $b_1 \leq b_2 \leq ... \leq b_n$ be the increasing representation of $\alpha$ then $\alpha$ is a parking function if and only if $b_i \leq i,$ and every permutation of the entries of a parking function is also a parking function.
\end{corollary}
\noindent It is obvious that if $\alpha = (a_1, a_2, a_3, ..., a_n)$ is a parking function on $n$ \emph{parking spaces} then $\alpha^* = (0, a_1, a_2, a_3, ..., a_n)$ is a parking function on $(n+1)$ \emph{parking spaces.} 
\section{Parking functions in respect of the degree sequences of simple connected graphs}
It is easy to see that if the vertices of a path $P_n, n \in \Bbb N$ are labelled from \emph{left to right} consecutively as $v_1, v_2, v_3, ..., v_n$, the degree sequence of the path $P_n,$ namely $(1, \underbrace{2, 2, ... 2,}_{(n-2)-entries} 1)$ is a parking function. Similarly, easy to see that the degree sequence of the cycle $C_n,$ namely $(\underbrace{2, 2, 2, ..., 2}_{n-entries})$ is not a parking function since exactly one vertex $v_i, i \in\{1, 2, 3, ..., n\}$ will \emph{cliff}.\\ \\
This leads to the observation that if the degree sequence of a spanning subgraph say, graph $H$ (or at least a spanning tree thereof, [3] and Theorem 2.6) of a simple connected graph $G$ is a parking graph, the degree sequence of $G$ might not be. However, the converse holds.
\begin{theorem}
If the degree sequence of a simple connected graph $G$ is a parking function then the degree sequence of any spanning subgraph $H$ of $G$ is a parking function.
\end{theorem}
\begin{proof}
Let a spanning subgraph of $G$ be graph $H$. Label the vertices of $G$ as $v_1, v_2, v_3, ..., v_n,$ and assume the degree sequence of $G$ is a parking function. Then from the definition of a spanning subgrah $H$ we have that the degree sequence is given by $(d_H(v_1)\leq d_G(v_1), d_H(v_2)\leq d_G(v_2), d_H(v_3)\leq d_G(v_3), ..., d_H(v_n)\leq d_G(v_n)).$ This implies that the degree sequence is merely a permutation of the increasing representation of the sequence $(d_H(v_i) \leq d_G(v_i) \leq d_H(v_{i+1}) \leq d_G(v_{i+1}), i = 1, 2, ..., n-1).$\\ \\
Since, $d_G(v_i) \leq i$ it follows that $d_H(v_i) \leq i$ and therefore a parking function.
\end{proof}
\noindent From our earlier observation that $\alpha^* = (0, a_1, a_2, a_3, ..., a_n)$ is a parking function on $(n+1)$ parking spaces if and only if $\alpha =(a_1, a_2, a_3, ..., a_n)$ is a parking function on $n$ parking spaces, it follows that if the degree sequence of a simple connected graph $G$ is a parking function on $n$ parking spaces , then the degree sequence of $G \cup mK_{1,_{m \rightarrow \infty}} $ is a parking function on $(n+m),_{m \rightarrow \infty}$ parking spaces.
\begin{proposition}
For a simple connected graph $G$ on n vertices the degree sequence $d(G)$ is a parking function if $\Delta(G) \leq \lceil\frac{n}{2} \rceil.$ 
\end{proposition}
\begin{proof}
Consider the graph on \emph{one} vertex $v_1.$ Hence, $d(v_1) = 0 \leq  \lceil\frac{1}{2}\rceil.$ With one \emph{parking space} available vertex $v_1$ can by default, park. So the result holds for $n=1.$ Assume it holds for any simple connected graph on $n = k$ vertices. Now consider any simple connected graph $G^* = G + {v_iv_{k+1}},$ for possibly multiple $i=1,2,3, ..., k.$ such that $d(v_{k+1}) \leq \lceil\frac{k+1}{2}\rceil.$ So now, we provide $k+1$ \emph{parking spaces.}\\ \\
If the $n$ vertices of  $G$ stream in at random as before they can occupy parking spaces as before or a permutation thereof and any one vertex may occupy the parking space $k+1$ as a result of a change in preferred value as well. At all times one parking space $p_i, i \geq \lceil\frac{k+1}{2}\rceil$ is available. Hence, vertex $v_{k+1}$ with $d(v_{k+1}) \leq \Delta (G) \leq \lceil \frac{n}{2}\rceil \leq\lceil\frac{n+1}{2}\rceil$ always has a \emph{parking space} to occupy. Through mathematical induction it follows that for all graphs for which, $\Delta(G) \leq \lceil\frac{n}{2}\rceil,$ the degree sequence $d(G)$ is a parking function.\\ \\
\end{proof}
\begin{corollary}
If the degree sequence of a simple connected graph $G$ is a parking function then $G$ has at least one pendant vertex.
\end{corollary}
\begin{proof}
Since a vertex parked in \emph{parking space} $p_1$ the result follows from the definition of connectivity and that of a parking function.\\ \\
\end{proof}
\begin{lemma}
For the Jaco Graph $J_n(1)$ we have that $\Delta(J_n(1)) \leq \lceil\frac{n}{2}\rceil.$
\end{lemma}
\begin{proof}
Because $\lceil\frac{n}{2}\rceil \leq \lceil\frac{n+1}{2}\rceil$, the result follows from Lemma 1.2 and Corollary 1.3 in $[10].$
\end{proof}
\begin{corollary}
The degree sequence of a Jaco Graph $J_n(1)$ is a parking function.
\end{corollary}
\begin{proof}
Follows immediately from the definition of a Jaco Graph $J_n(1)$ in [10] and Proposition 2.2 above.
\end{proof}
\begin{theorem}
The degree sequence of any tree is a parking function.
\end{theorem}
\begin{proof}
Consider the tree on \emph{one} vertex $v_1.$ Hence, $d(v_1) = 0.$ With one \emph{parking space} available vertex $v_1$ can be default, occupy. So the result holds for $n=1.$ Assume it holds for any tree on $n = k$ vertices. Now consider any tree $T^* = T + {v_iv_{k+1}}, i \in \{1,2,3, ..., k\}.$ So we provide $k+1$ \emph{parking spaces}. If the $n$ vertices of  $T$ stream in at random as before they can occupy parking spaces as before or a permutation therefore. Any vertex may even occupy parking space $k+1$ as a result of change of preferred value as well. At all times one parking space $p_i, i \geq 1$ is available. Since $d(v_{k+1}) = 1$ it can occupy any vacant parking space $p_i, i \geq 1.$ Hence the result follows for any tree on $n = k+1$ vertices and therefore, through mathematical induction it follows to hold for all trees on $n \in \Bbb N$ vertices.
\end{proof}
\begin{proposition}
Consider any tree $T$ on $n \geq 4$ vertices and construct the tree $T^*$ by linking any two vertices $v_i, v_j$ of $T$ such that at least one pendant vertex remains. Then the degree sequence of $T^*$ is a parking function.
\end{proposition}
\begin{proof}
Construct any $T^*$ as defined in the proposition. Label the cycle $C$ and assume it is isomorphic to $C_m$. So $T^* - C_m$ is a forest of $\ell$ trees, $F = \cup_i^\ell T_i$. Consider any tree $T_j \in F$ which was linked to $v_i$ in $C_m$. Link one pendant $v_{k_1}$ to $v_i$. Because all degrees of all vertices in $C_m$ equals 2 only one vertex cliffs because all vertices have to skip parking $p_1$. But with the added pendant vertex an additional parking space $p_{m+1}$ is allowed so all vertices of $C_m$ \emph{park} and the pendant vertex occupies $p_1$. Hence the degree sequence of $C_m + v_{k_1}$ is a parking function.\\ \\
By linking another pendant $v_{k_2}$ to $v_i$ or $v_{k_1}$ a similar result follows because now we provide $m+2$ parking spaces. By smartly linking pendants recursively, the tree $T_j$ can be reconstructed as linked before to $v_i$. So it follows that the degree sequence of $C_m + T_j$ is a parking function. Label $C_m + T_j$ as $G^*_{m+1}$. So we proved the result for $G^*_{m+1}$.\\ \\
Assume the result holds for $G^*_{m+x}, x \leq \ell$. Also assume that $G^*_{m+x} = C_m+ (T_1 + T_2 +, ..., + T_x)$ has $(m+t)$ vertices. This means that the $(m+t)$ vertices all park in $(m+t)$ parking spaces. Consider any other tree $T_{x+1} \in F$ which was linked to $v^*_i$. By the same reasoning as before the degree sequence of $G^*_{m+x} + v^*_i$ is a parking function. So through finite mathematical induction the result follows.
\end{proof}
\begin{corollary}
The degree sequence of an unicyclic graph is a parking function.
\end{corollary}
\begin{proof}
Follows directly from Proposition 2.7.
\end{proof}
\noindent Although we mainly consider simple connected graphs it is useful to note that:
\begin{lemma}
If the respective degree sequences of simple connected graphs $G_1, G_2, G_3, ..., G_s$ are all parking functions the \emph{the degree sequence} of the graph $H = G_i \cup_{\forall j\neq i}^s G_j,$  is a parking function.
\end{lemma}
\begin{proof}
Let $|V_1| = \nu_1, |V_2| = \nu_2, ..., |V_s| = \nu_s.$ Since $G_\ell \cup G_q = G_q \cup G_\ell$ we have that $H = G_i \cup_{\forall j\neq i}^s G_j \simeq G_1 \cup G_2 \cup ... \cup G_s.$ Without loss of generality consider graphs $G_1$ and $G_2$ and the degree sequence of $G_1$ stream in first. Note that $\nu_1 + \nu_2$ parking spaces are available. Because the degree sequence of $G_1$ is a parking function all the vertices of $G_1$ will park within the first $\nu_1$ parking spaces. Now $\nu_2$ parking spaces remain and numbered $(\nu_1+1), (\nu_1 +2), (\nu_1+3), ..., (\nu_1+\nu_2).$\\ \\
Since the degree sequence of $G_2$ is a parking function a vertex $v_i$ of $G_2$ initially occupied a parking space $p_j, 1 \leq p_j \leq \nu_2.$ The equivalent parking space in $H_2 = G_1 \cup G_2$ that is available is numbered, $p_j + \nu_1 \leq \nu_1 + \nu_2.$ Hence all vertices of $G_2$ will find parking space. So the result holds for $H_2 = G_1 \cup G_2.$\\ \\
Assume it holds for $H_t = G_1 \cup G_2 \cup ... \cup G_t.$ Now consider the graph $H_{t+1} = G_1 \cup G_2 \cup ... \cup G_t \cup G_{t+1}.$  Since, $H_{t+1} = G_1 \cup G_2 \cup ... \cup G_t \cup G_{t+1} = (G_1 \cup G_2 \cup ... \cup G_t) \cup G_{t+1},$ the results holds for the graph $H_{t+1}$ and it follows immediately to hold in general.
\end{proof}
\begin{definition}
For a simple connected graph $G$, define the set of derivative degree sequences,\\ \\ $\Bbb D_d(G) = \{(\lceil\frac{d(v_1)}{\ell}\rceil, \lceil\frac{d(v_2)}{\ell}\rceil, \lceil\frac{d(v_3)}{\ell}\rceil, ..., \lceil\frac{d(v_n)}{\ell}\rceil), \ell = d(v_i), \forall i,$ with $d(v_i) \geq 2\}.$
\end{definition}
\noindent The next theorem is of importance.
\begin{theorem}(Daneel's theorem)\footnote{In memory of my young friend who so untimely (11 Spetember 2013, age 25), parked his soul somewhere in space.}
A derivative degree sequence $d_d(G) \in \Bbb D_d(G)$ is a parking function.
\end{theorem}
\begin{proof}
The derivative degree sequence with \emph{largest} entries is certainly, $(\lceil\frac{d(v_i)}{min(d(v_i))_{d(v_i)\geq2}}\rceil, i= \\ \\1, 2, 3, ..., n).$ This implies that if $(\lceil\frac{d(v_i)}{min(d(v_i))_{d(v_i)\geq2}}\rceil, i=1, 2, 3, ..., n),$ can be shown to be \\ \\a parking function then all others derivative degree sequences are as well since, $\lceil\frac{d(v_i)}{\ell}\rceil \leq \\ \\ \lceil\frac{d(v_i)}{min(d(v_i))_{d(v_i)\geq2}}\rceil,\ell = d(v_i) \geq 2,\forall i.$ (See Corollary 1.1).\\ \\
For any simple connected graph we have that if $\lceil\frac{d(v_i)}{min(d(v_i))_{d(v_i)\geq2}}\rceil$ for exactly one vertex\\ \\then $\Delta(G) \leq (n-2).$ So in general $\Delta(G) \leq (n-2)$ for all simple connected graphs with\\ \\ $\lceil\frac{d(v_i)}{min(d(v_i))_{d(v_i)\geq2}}\rceil \geq 2.$\\ \\
The above implies that if $n$ is even then $(n+2)$ \emph{parking spaces} are available beyond the largest \emph{preferred value,} $\lceil\frac{n-2}{2}\rceil.$ So in all cases we have that $\lceil\frac{d(v_i)}{min(d(v_i))_{d(v_i)\geq2}}\rceil \leq i, i=1, 2, 3, ..., n,$ which implies the derivative degree sequence to be a parking function.\\ \\
It also implies that if $n$ is uneven then $(n+1)$ \emph{parking spaces} are available beyond the largest \emph{preferred value,} $\lceil\frac{n-2}{2}\rceil.$\\ \\ So in all cases we have that $\lceil\frac{d(v_i)}{min(d(v_i))_{d(v_i)\geq2}}\rceil \leq i, i=1, 2, 3, ..., n,$ which implies the derivative degree sequence to be a parking function.\\ \\
Finally then, since $(\lceil\frac{d(v_i)}{min(d(v_i))_{d(v_i)\geq2}}\rceil, i=1, 2, 3, ..., n)$ is a parking function for both \emph{n even or uneven,} the result that a derivative degree sequence $d_d(G) \in \Bbb D_d(G)$ of a simple connected graph $G$ is a parking function, follows.
\end{proof}
\subsection{Looping degree sequences of $G-\{v_1, v_2, v_3, ..., v_r\}, r \leq n-1$}
\emph{Looping degree sequences} are found by allowing the degree sequence of a simple connected graph $G$ to stream in and when the maximum set of vertices $T=\{v_t|v_t$ \emph{parks}$\}$ parks, the degree sequence reduces to the degree sequence of $G-T.$ Recursively then, when on the first stream vertex $v_r$ parks, the degree sequence reduces to the degree sequence of $G-\{v_1, v_2, v_3, ..., v_r\}, r \leq n-1.$ Now the first \emph{loop} streams with the degree sequence of $G-T.$ After all vertices find parking we refer to the set $\Bbb R_d$ of recursive degree sequences as a \emph{recursive parking function}.\\ \\
\textbf{Example:} For the complete graph $K_n, n$ even the \emph{recursive parking function} is given by the set $\Bbb R_d = \{\underbrace{(n-1, n-1, n-1, ..., n-1)}_{n-entries}, \underbrace{(n-3, n-3, n-3, ..., n-3)}_{(n-2)-entries}, ..., (1, 1)\}.$ Also see Lemma 2.12 (Case 1). 
\begin{theorem}
For any simple connected graph $G$ on $n$ vertices the recursive parking function is a finite set.
\end{theorem}
\begin{proof}
It is obvious that if the degree sequence of a simple connected graph $G$ is a parking function, the result holds since such a degree sequence is a recursive degree sequence in itself.\\ \\
Furthermore, it is obvious that we only have to consider the extermal case where it is found that on the first stream and subsequent looping streams, only one vertex at a time finds parking. All other cases show an improvement in that the number of loops required decreases. It is also known that if the degree sequence of a simple connected graph $G$ is not a parking function then, neither the permutations thereof are parking functions. However, any first vertex of the degree sequence stream or of a permutation thereof, will park since $n$ parking spaces are available and $\Delta(G) \leq n-1,$ and therefore any $d(v_i) \leq n-1, \forall i.$\\ \\
Therefore, on the next round of streaming (\emph{looping}), $n-1$ parking spaces are available whilst $\Delta(G-v_1) \leq n-2.$ This means that any second vertex of the \emph{recursive degree sequence} or of a permutation thereof, can park. After $n-1$ recursions a single \emph{(last)} vertex remains say, $v_k$ with $d(v_k) = 0,$ and exactly \emph{one} parking space is left. It means that by the default preferred value convention and the parking rule, vertex $v_k$ may park at any available \emph{parking space}, $p_i, 1 \leq i \leq n$ which is not occupied. It also means that for any simple connected graph $G$ the set of recursive degree sequences or any permutation thereof (\emph{implicitly all} $(n+1)^{n-1}$ \emph{cases covered}), are a finite recursive parking function because all vertices necessary find parking.
\end{proof}
\subsection{Looping number $\xi(G)$ of a simple connected graph, $G$}
The \emph{looping number} $\xi(G)$ of a simple connected graph $G$ is the \emph{maximum} number of loops required for all vertices to park. Clearly, $\xi(G) = 0$ if and only if the degree sequence of $G$ is a parking function. 
\begin{lemma}
For a simple connected graph $G$ on n vertices we have, $0 \leq \xi(G) \leq \lfloor\frac{n-1}{2}\rfloor.$
\end{lemma}
\begin{proof}
If the degree sequence of a simple connected graph $G$ is a parking function, then $\xi(G) = 0.$\\ \\
So we consider the cases where the degree sequences of $G$ are not parking functions. Since, the degree sequence of $G - v_i$ equals the degree sequence of $(G + v_iv_j)_{j \neq i, v_j \in V(G)} - v_i,$ it follows that $\xi(G) \leq \xi(K_n).$ Hence $\xi(K_n)$ provides the \emph{upper bound}.\\ \\
\noindent Case 1: Let $n$ be even. It follows that in $K_n$ we have $d(v_i) = n-1, \forall i.$ Hence, on first streaming exactly two vertices can park in $p_{n-1}$ and $p_n$ whilst the $n-2$ vertices of $K_{n-2}$ loops (\emph{first loop}). Thereafter, exactly two more vertives (\emph{degree} = $(n-2) - 1)$ can park in $p_{n-3}$ and $p_{n-2}$ whilst the $n-4$ vertices of $K_{n-4}$ loops (\emph{second loop}). Recursively after exactly $\frac{n}{2} -2$ \emph{loops} only $K_2$, with exactly two parking spaces $p_1, p_2$ remain. Hence, after one more loop or, after $(\frac{n}{2} -2) + 1 = \lfloor\frac{n-1}{2}\rfloor$ \emph{loops} all vertices park.\\ \\
\noindent Case 2: Let $n$ be uneven. The proof follows similar to Case 1 except that recursively after exactly $\frac{n}{2} -2$ \emph{loops}, only $K_1$, with exactly one parking space $p_1$ remain. Hence, after one more loop or, after $(\frac{n}{2} -2) + 1 = \lfloor\frac{n-1}{2}\rfloor$ \emph{loops} all vertices park.
\end{proof}
\noindent Note that if the respective degree sequences of graphs $G_1$ and $G_2$ are parking functions the degree sequence of $G_1 + G_2$ is not necessarily a parking function. It remains an open problem for which graphs the result will hold.
\begin{corollary}
We have that $\xi (\cup_{i\in \Bbb N} P_i) = \xi (P_{\sum\limits_{i\in \Bbb N}}).$
\end{corollary}
\begin{proof}
Consider $P_n$ and $P_m$, $(n\geq 1) \in\Bbb N, (m\geq 1) \in \Bbb N.$ The vertices of path $P_n$ can park in the first $n$ parking spaces and so can the vertices of $P_m$ park in the following $n+1$ to $n+m$ parking spaces or vice versa. Equally so can the vertices of $P_{n+m}$ park in $(n+m)$ parking spaces because we have in $P_{n+m}$ that, $d(v_n) =2$ and $d(v_{n+1}) =2.$ Hence, these two vertices can still park in parking spaces, $n$ and $n+1.$ So the result follows for $P_n$ and $P_m$. Through induction it follows that $\xi (\cup_{i\in \Bbb N} P_i) = \xi (P_{\sum\limits_{i\in \Bbb N}}).$ 
\end{proof}
\noindent Now the next theorem can be settled.
\begin{theorem}
For the respective degree sequences of two simple connected graphs $G_1$ and $G_2$ we have that:\\ \\
(a) if both degree sequences are parking functions, the looping number $\xi(G_1 + G_2) \leq 1$,\\
(b) if at least one degree sequence is not a parking function, the looping number $\xi(G_1 + G_2) \leq \xi(G_1) + \xi(G_2) + 1.$
\end{theorem}
\begin{proof}
\noindent (a) Consider the simple connected graphs $G_1$ and $G_2$ on $n$ and $m$ vertices, respectively. Assume both degree sequences of $G_1$ and $G_2$ are parking functions. Because $G_1 + G_2 = G_2 + G_1$ we only have to consider the case $G_1 + G_2.$ Without loss of generality assume $n \leq m.$ \\ \\
Case a(1). Assume that all the vertices of graph $G_1$ stream in randomly, first. Then as before these vertices will park in the consecutively labelled parking spaces $p_{1+m}, p_{2+m}, ..., p_{n+m}.$ Now the vertices of $G_2$ stream in randomly and since the remaining parking spaces are labelled $p_1, p_2, ..., p_n, p_{(n+1)}, p_{(n+2)}, ..., p_{n+(m-n)}$ some vertices of $G_2$ can park by definition, and some (or all) may \emph{cliff} to enter loop one.  On the first loop the graph $G_1+G_2$ reduces to at most $G_2$ or a subgraph of $G_2$ and since the number vacant parking spaces equals the number of vertices $v_i \in V(G_2)$ that cliffed, those vertices can all park by definition. It implies that $\xi(G_1+ G_2) = 1.$ Since no cliffing occurs for $K_1 + K_1$ the result suggests that if both the respective degree sequences of two simple connected graphs $G_1$ and $G_2$ are parking functions, then $\xi(G_1 + G_2) \leq 1.$\\ \\
Case a(ii) Assume that all the vertices of graph $G_2$ stream in randomly, first. Then as before these vertices will park in the consecutively labelled parking spaces $p_{1+n}, p_{2+n}, ..., p_{m+n}.$  Since exactly $n$ parking spaces labelled, $p_1, p_2, ..., p_n$ remain vacant and only the vertices of $G_1$ with the reduces degree sequence of $G_1$ streams in on loop 1, all vertices can park. As before we have that for $K_1 + K_1$ no cliffing occurs so the result suggests that if both the respective degree sequences of two simple connected graphs $G_1$ and $G_2$ are parking functions, then $\xi(G_1 + G_2) \leq 1.$\\ \\
Case a(iii) Also let any vertex say, $v_k \in V(G_1)$ stream in first. As an extremal case assume vertex $v_k$ parked in the space $p_n$ initially. Now the degree of $v_k$ increased by $m$ so it can park in, at most, the parking space $p_{n+m}$ in $G_1 + G_2.$ Now as the next extremal case assume without loss of generality that any vertex say, $v_\ell \in V(G_2)$ streams in. Certainly then, since the degree of vertex $v_\ell \in V(G_2)$ has increased by $n$ it can park in, at most, the vacant parking space numbered $(m-n)$. So this is possible for all vertices of the graph $G_1 + G_2$ streaming in randomly, allowing for some vertices of both $V(G_1)$ and $V(G_2)$ to cliff. However, on streaming loop 1, all vertices park because the degree values decrease by either $n$ or $m.$ As before we have that for $K_1 + K_1$ no cliffing occurs so the result suggests that if both the respective degree sequences of two simple connected graphs $G_1$ and $G_2$ are parking functions, then $\xi(G_1 + G_2) \leq 1.$\\ \\
Since all cases have been argued and all the suggestions are assertive, the partial  result that for some graphs $\xi(G_1 + G_2) < \xi(G_1) + \xi(G_2),$ follows conclusively.\\ \\  
\noindent (b)  We know that $\xi(K_n) = 1, \forall n \in \Bbb N.$ Since $K_{n+m} = K_n + K_m$ we have that $\xi(K_{n+m}) = 1 < \xi(K_n) + \xi(K_m)$ so the inequality holds. We also know that $\xi(C_n) = 1, \forall n \in \Bbb N.$ For $C_n + C_m$ we have that $d_{C_n}(v) =2 + m,$ $ \forall v \in V(C_n)$ and $d_{C_m}(u) =2 + n,$ $ \forall u \in V(C_m).$ Also assume without loss of generality that $n \geq m.$\\ \\
Case (b)(i). Let the vertices of $C_n$ stream in first. It implies that exactly $n-1$ vertices of $C_n$ can park in \emph{parking spaces} $p_{(2+m)}, p_{(2+m) + 1},..., p_{(n+m)}$. On loop one we have the graph $K_1 + C_m$ looping. Since the graph has no pendant vertex a further loop will be required, leaving exactly $P_2$ looping into parking. Hence $\xi(C_n + C_m) = 2.$ It suggests that for some graphs equality $\xi(G_1 + G_2) = \xi(G_1) + \xi(G_2)$,  holds.\\ \\
Case (b)(ii). Let the vertices of $C_m$ stream in first. It implies that exactly $m-1$ vertices of $C_m$ can park with exactly one vertex of $C_n$ parking as well. On loop one we have the graph $K_1 + P_{n-1}$ looping. Since the graph has no pendant vertex a further loop will be required. Considering all random streaming we are left with either $K_1$ and one vacant parking space, or $K_1 \cup K_1$ and two vacant parking spaces or $P_2$ and two vacant parking spaces.  It suggests that for some graphs, equality $\xi(G_1 + G_2) = \xi(G_1) + \xi(G_2)$, holds. \\ \\
Since all cases have been argued and all the suggestions are assertive, the partial result that for some graphs $\xi(G_1+ G_2) = \xi(G_1) + \xi(G_2)$, follows conclusively.\\ \\  
\noindent (c) Consider the \emph{tetrahedron}, $G_1,$ [3] and the path $P_n$. We have that $d_{G_1}(v) = 3+n,$ $\forall v \in V(G_1)$ and $d_{P_n}(u) = 5$ (pendant vertices) or $6$. Any random streaming of vertices allows exactly $n$ vertices to park and on stream one, a \emph{tetrahedron} loops. Hence loop two is required because the tetrahedron has no pendant vertices. It follows that for some graphs, $\xi(G_1 + P_n) = 2 = (0 + 1) + 1 = \xi(G_1) + \xi(P_n) + 1.$ \\ \\
So we could show for specific cases that $\xi(G_1 + G_2) \leq \xi(G_1) + \xi(G_2) + 1.$ \\ \\
To settle the theorem we need to show that $\xi(G_1 + G_2) > \xi(G_1) + \xi(G_2) + 1,$ is false in general. In terms of the definition of $\xi(G)$ we know that $K_n, n \in \Bbb N$ is the most complex graph and we know that $G_1 + G_2$ is always a subgraph of $K_{(n+m)}$. It easily follows that $\xi(K_s) = \xi (K_{(n+m)}) = \xi(K_n + K_m) =\xi(K_n)+ \xi(K_m).$ In fact the $"+1"$ only follows if the degree sequence of only one graph is a parking function. Hence, $\xi(G_1 + G_2) > \xi(G_1) + \xi(G_2) +1$ is false in general.\\ \\
So the result of the theorem follows.
\end{proof}
\subsection{Appendix III of Bondy and Murty $[3]$}
As stated in Bondy and Murty $[3]$, there are a number of graphs which are interesting. We will present the looping number of some of those mentioned.\\ \\
\textbf{2.3.1 Frucht graph ([8], 1949)}\\ \\
For the \emph{Frucht graph,} $F_1$ on first stream, any 10 of the 12 vertices streaming at random will park and always leave two isolated (disjoint) vertices say, $v_i$ and $v_j$. Since $d(v_i) = d(v_j) = 0$ they may both occupy any of the two remaining vacant parking spaces on loop one. Hence, $\xi(F_1) = 1.$\\ \\ \\
\textbf{2.3.2 Folkman graph ([6], 1967)}\\ \\ 
Folkman proved that every edge, but not vertex-transitive regular graph, has at least twenty vertices. The Folkman graph $F_2$ has exactly twenty vertices, the best possible result. Each vertex $v_i$ has $d(v_i) = 4.$ So a set of $|T|= 17$ vertices streaming randomly on the first stream will park leaving three vertices in $F_2-T$ of which at most, only one vertex say, $v_i$ will have $d(v_i) = 2$ and the other two with degree 0. Hence, on loop one all vertices will park. So, $\xi(F_2) = 1.$\\ \\
\textbf{2.3.3 The platonic octahedron graph ([7], 1967)}\\ \\
The graph is 4-regular on six vertices. Hence any three vertices can park on the first stream leaving either $C_3$ or $P_3$ to loop. Since parking spaces $p_1, p_2$ and $p_3$ are available only two vertices can park on loop one in the case $C_3,$ loops. So, $\xi = 2.$\\ \\
We further observe that $\Bbb R_d = \underbrace{\{(4, 4, 4, 4, 4, 4), (2, 2, 2), (0)\}}_{C_3-loops}$ or $\underbrace{\{(4, 4, 4, 4, 4, 4), (1, 2, 1)\}}_{P_3-loops}.$ It shows that the \emph{looping number} is dependent on the permutations of vertex streaming per loop. \\ \\
\noindent [Open problem: For which graphs will we have that, if the respective degree sequences of graphs $G_1$ and $G_2$ are parking functions then the degree sequence of $H = G_1 + G_2$ is a parking function as well? $K_1 + K_1 = P_2$ is an example.] \\  \\
\noindent [Open problem: Consider the simple connected graphs $G_1, G_2, G_3, ..., G_n.$ What can be said about $\xi(\cup_{(1 \leq i \leq n)} G_i)?$]\\ \\
\noindent [Open problem: If for the simple connected graphs $G$ and $H_i, i =1, 2, 3, ..., k$ at least one degree sequence is not a parking function, it is expected that the looping number, $\xi(G+ H_{i, \forall i}) \leq \xi(G) + \sum\limits_{\forall i} \xi(H_i) +1.$ Is the conjecture true?]\\ \\
\noindent [Open problem: Define the first line graph of a simple connected graph $G$ the graph $G^{\ell=1}$. We know that the degree sequence of $P_n$ is a parking function. We also know that $P^{\ell \rightarrow \infty}_n \rightarrow P_1$ of which the degree function is a parking function. We also know that $C^{\ell \rightarrow \infty}_n \rightarrow C_n$ of which the degree sequence is not a parking function. \\ \\
(a) If it is true that the degree sequence of a simple connected graph $G$ is a parking function, is it consequently true that the degree sequence of the line graph $G^{\ell \rightarrow \infty}$ is a parking function as well ?\\ \\ 
(b) If it is true that the degree sequence of a simple connected graph $G$ is not a parking function, is it consequently true that the degree sequence of the line graph $G^{\ell \rightarrow \infty}$ is not a parking function as well ?]\\ \\
\noindent \textbf{\emph{Open access:\footnote {To be submitted to the \emph{Pioneer Journal of Mathematics and Mathematical Sciences.}}}} This paper is distributed under the terms of the Creative Commons Attribution License which permits any use, distribution and reproduction in any medium, provided the original author(s) and the source are credited. \\ \\ \\ \\
References (Limited) \\ \\
$[1]$ Aker, K., Can, M.B., \emph{From parking functions to Gelfand pairs}, Proceedings of the American Mathematical Society, Vol 140 (2012), no. 4, pp 1113-1124. \\
$[2]$ Armstrong, D., Garsia, A., Haglund, J., Rhoades, B., Sagan, B., \emph{Combinatorics of Tesler matrices in the theory of parking functions and diagonal harmonics,} Journal of Combinatorics, Vol 3 (2012), no.3, pp 451-494. \\
$[3]$ Bondy, J.A., Murty, U.S.R., \emph {Graph Theory with Applications,} Macmillan Press, London, (1976). \\
$[4]$ Chebikin, D., Postnikov, A., \emph{Generalised parking functions, descent numbers, and chain polytopes of ribbon posets}, Advances in Applied Mathematics, Vol 44 (2010), no. 2, pp 145-154.\\  
$[5]$ Dotsenko, V., \emph{Parking functions and vertex operators}, Selecta Mathematica, Vol 14 (2009), no. 2, pp 229-245.\\
$[6]$ Folkman, J., \emph{Regular line-symmetric graphs,} Journal of Combinatorial Theory, Vol 3 (1967), pp 215-232.\\
$[7]$ Frechet, M., Fan, K., \emph{Initiation to Combinatorial Topology,} Prindle, Weber and Schmidt, Boston, (1967).\\
$[8]$ Frucht, R., \emph{Graphs of degree three with a given abstract group,} Canadian Journal of Mathematics, Vol 1 (1949), pp 365-378.\\
$[9]$ Hopkins, S., Perkinson, D., \emph{Bigraphical arrangements}, arXiv: 1212.4398v2 [math.CO], 2012.\\
$[10]$ Kok, J., Fisher, P., Wilkens, B., Mabula, M., Mukungunugwa, V., \emph{Characteristics of Finite Jaco Graphs, $J_n(1), n \in \Bbb N$}, arXiv: 1404.0484v1 [math.CO], 2 April 2014. \\
$[11]$ Shin, H., \emph{A New Bijection Between Forests and Parking Functions}, arXiv: 0810.0427v2 [math.CO], 2010.\\
$[12]$ Stanley, R.P., \emph{Parking Functions}. Department of Mathematics, M.I.T., Cambridge, MA 02139.\\
\end{document}